\documentclass[a4paper,11pt]{amsart}
\usepackage{amssymb,amscd,amsxtra,xypic}
\usepackage[all]{xy}
\usepackage{array}

\usepackage[pdftex]{hyperref}
\hypersetup{%
bookmarksnumbered=true,%
colorlinks=true,%
setpagesize=false,%
pdftitle={},%
pdfauthor={Takuzo Okada}}

\theoremstyle{plain}
\newtheorem{Thm}{Theorem}[section]
\newtheorem{Lem}[Thm]{Lemma}

\newtheorem{Conj}[Thm]{Conjecture}

\theoremstyle{definition}
\newtheorem{Def}[Thm]{Definition}
\newtheorem{Def-Lem}[Thm]{Definition-Lemma}

\newtheorem{Rem}[Thm]{Remark}

\newtheorem*{Ack}{Acknowledgments}

\newcommand{\Aut}{\operatorname{Aut}}
\newcommand{\Bir}{\operatorname{Bir}}

\newcommand{\prt}{\partial}
\newcommand{\Sing}{\operatorname{Sing}}
\newcommand{\Spec}{\operatorname{Spec}}
\newcommand{\Cr}{\operatorname{Cr}}

\newcommand{\rank}{\operatorname{rank}}

\newcommand{\Cl}{\operatorname{Cl}}

\newcommand{\Pic}{\operatorname{Pic}}

\newcommand{\PSL}{\mathrm{PSL}}

\newcommand{\Cox}{\mathrm{Cox}}

\newcommand{\mbA}{\mathbb{A}}

\newcommand{\mbC}{\mathbb{C}}
\newcommand{\mbF}{\mathbb{F}}

\newcommand{\mbP}{\mathbb{P}}
\newcommand{\mbQ}{\mathbb{Q}}

\newcommand{\mbZ}{\mathbb{Z}}

\newcommand{\mcL}{\mathcal{L}}
\newcommand{\mcM}{\mathcal{M}}

\newcommand{\mcO}{\mathcal{O}}

\newcommand{\mcQ}{\mathcal{Q}}

\newcommand{\mcX}{\mathcal{X}}

\newcommand{\K}{\Bbbk}

\newcommand{\inj}{\hookrightarrow}

\newcommand{\ratmap}{\dashrightarrow}

\newcommand{\bmu}{\boldsymbol{\mu}}

\makeatletter
\def\imod#1{\allowbreak\mkern10mu({\operator@font mod}\,\,#1)}
\makeatother

\title[Nonrational del Pezzo fibrations]{Nonrational del Pezzo fibrations admitting an action of the Klein simple group}
\author[Takuzo Okada]{Takuzo Okada}
\address{Department of Mathematics, Faculty of Science and Engineering\endgraf
Saga University, Saga 840-8502 Japan}
\email{okada@cc.saga-u.ac.jp}
\date{}

\begin{document}

\begin{abstract}
We present a series of del Pezzo fibrations of degree $2$ admitting an action of the Klein simple group and prove their nonrationality by the reduction modulo $p$ method of Koll\'{a}r.
This is relevant to embeddings of the Klein simple group into the Cremona group of rank $3$.
\end{abstract}

\maketitle


\section{Introduction} \label{sec:intro}

The Klein simple group $G_K$ is a finite simple group $G_K \cong \PSL_2 (\mbF_7)$ of order $168$ and it is well known that $G_K$ is the automorphism group of the Klein quartic curve which is defined in $\mbP^2$ by the equation $x_0^3 x_1 + x_1^3 x_2 + x_2^3 x_0 = 0$.
Let $S_K$ be the double cover of $\mbP^2$ ramified along the Klein quartic curve.
Then $S_K$ is a nonsingular del Pezzo surface of degree $2$ admitting a faithful action of $G_K$.
Belousov \cite{Bel} proved that $\mbP^2$ and $S_K$ are the only del Pezzo surfaces admitting a faithful action of $G_K$.
In \cite{Ahm1}, Ahmadinezhad presented a series of $G_K$-Mori fiber spaces $X_n/\mbP^1$ over $\mbP^1$ whose general fibers are isomorphic to $S_K$ for $n \ge 0$.
A $G$-Mori fiber space, where $G$ is a group, is a $G$-equivariant version of Mori fiber space (see Definition \ref{def:GMFS}).
Among the above series of varieties, $X_n/\mbP^1$ is a del Pezzo fibration for $n \ge 1$ while $X_0/\mbP^1 = \mbP^1 \times S_K/\mbP^1$ is not (see Section \ref{sec:constdP} for details).
We have the following conjectures concerning these varieties.

\begin{Conj}[{\cite{Ahm1}}] \label{conj:KleindP}
The fibrations $\mbP^1 \times \mbP^2/\mbP^1$ and $X_n/\mbP^1$, for $n \ge 0$, are the only $G_K$-Mori fiber spaces over $\mbP^1$ in dimension $3$.
\end{Conj}

\begin{Conj}[{\cite{Ahm1}}] \label{conj:nonrat}
The varieties $X_n$ are non-rational for $n \ge 2$.
\end{Conj}

We note that $X_0$ and $X_1$ are both rational.
We state the main theorem which supports Conjecture \ref{conj:nonrat}.

\begin{Thm} \label{thm:main}
For $n \ge 5$, a very general  $X_n$ is not rational.
\end{Thm}

We refer the readers to Section \ref{sec:proof} for the meaning of very generality.
Note that $X_n/\mbP^1$ is a del Pezzo fibration of degree $2$ and it satisfies the so called $K$-condition (or $K^2$-condition) for $n \ge 2$.
Thus, in view of the results of Pukhlikov \cite{Puk} and Grinenko \cite{Gri00,Gri03,Gri06} on nonsingular del Pzeeo fibrations of degree $2$, if $X_n$ were nonsingular, then it would be birationally rigid for $n \ge 2$, which would imply nonrationality in a strong sense.
Unfortunately, $X_n$ is singular and we cannot apply the above results directly.
Instead, we apply the Koll\'{a}r's reduction modulo $p$ method introduced in \cite{Kol1} (see also \cite{Kol}) to prove nonrationality of $X_n$.

This is relevant to the study of embeddings of the Klein simple group $G_K = \mathrm{PSL}_2 (\mbF_7)$ into the Cremona group $\Cr_3 (\mbC)$ of rank $3$.
If we are given a finite simple subgroup $G$ of $\Cr_3 (\mbC^3)$, then there is a rational $G$-Mori fiber space $X/S$ such that the embedding $G \subset \Cr_3 (\mbC)$ is given by $G \subset \Aut (X) \subset \Bir (X) \cong \Cr_3 (\mbC)$ (see \cite[Section 4.2]{Prokhorov}).
Such a $G$-Mori fiber space $X/S$ is called a {\it Mori regularization} of $G \subset \Cr_3 (\mbC)$.
Moreover, two embeddings $G_1$ and $G_2$ into $\Cr_3 (\mbC)$ of a finite simple subgroup $G$ are conjugate if and only if there is a $G$-equivariant birational map between Mori regularizations $X_1/S_1$ and $X_2/S_2$ of $G_1 \subset \Cr_3 (\mbC)$ and $G_2 \subset \Cr_3 (\mbC)$, respectively.
In \cite{CS}, Cheltsov and Shramov proved that there are at least three non-conjugate embeddings of $\mathrm{PSL}_2 (\mbF_7)$ into $\Cr_3 (\mbC)$ and each of them comes from rational ($G_K$-)Fano threefolds.
Theorem \ref{thm:main} implies that, for $n \ge 5$, a very general $X_n/\mbP^1$ cannot be a Mori regularization of any subgroup of $\Cr_3 (\mbC)$ isomorphic to $G_K$. 
If Conjectures \ref{conj:KleindP} and \ref{conj:nonrat} are both true, then it follows that there is no embedding of $\mathrm{PSL}_2 (\mbF_7)$ into $\Cr_3 (\mbC)$ coming from $G_K$-Mori fiber space over $\mbP^1$ other than $\mbP^1 \times \mbP^2/\mbP^1$, $X_0/\mbP^1$ and $X_1/\mbP^1$.

We explain the content of this paper.
In Section \ref{sec:constdP}, we give an explicit construction of the varieties $X_n$.
They are constructed as hypersurfaces of suitable weighted projective space bundle over $\mbP^1$.
Then we show that $X_n/\mbP^1$ is indeed a del Pezzo fibration for $n \ge 1$.
In Section \ref{sec:proof}, we prove the main theorem.
Proof will be done by the Koll\'{a}r's reduction modulo $p$ method, which we briefly recall in Section \ref{sec:Kol}.
The very first reduction step is done in Section \ref{sec:redmod2}
In Section \ref{sec:biglb}, we work over a field of characteristic $2$ and  construct a specific big line bundle on some nonsingular model of $X_n$ by making use of the purely inseparable double covering structure.
This will complete the proof in view of the non-ruledness criterion given in Lemma \ref{lem:crinonruled}.

\begin{Ack}
The author would like to thank Dr.~Hamid Ahmadinezhad for useful conversations.
He also would like to thank Professor Ivan Cheltsov for showing his interest on this work.
The author is partially supported by JSPS KAKENHI Grant Number 26800019.
\end{Ack}

\section{Construction of del Pezzo fibrations}
\label{sec:constdP}

We construct del Pezzo fibrations $X_n/\mbP^1$ as hypersurfaces in suitable weighted projective space bundles over $\mbP^1$.
We refer the readers to \cite{BC} for Cox rings (which is also known as homogeneous coordinate rings) of toric varieties.
Through the present section, we work over $\mbC$.

Throughout this paper, we define $f = x_0^3 x_1 + x_1^3 x_2 + x_2^3 x_0$.
We see that $f$ is the defining polynomial of the Klein quartic curve whose automorphism group is the Klein simple groups.
Let $P_n$ be the projective simplicial toric variety with Cox ring 
\[
\Cox (P_n) = \mbC [w_0,w_1,x_0,x_1,x_2,y]
\] 
which is $\mbZ^2$-graded as
\[
\begin{pmatrix}
w_0 & w_1 & x_0 & x_1 & x_2 & y \\
1 & 1 & 0 & 0 & 0 & -n \\
0 & 0 & 1 & 1 & 1 & 2
\end{pmatrix}
\]
and with the irrelevant ideal $I = (w_0,w_1) \cap (x_0,x_1,x_2,y)$, that is, $P_n$ is the geometric quotient
\[
P_n = (\mbA^6 \setminus V (I)) / (\mbC^*)^2,
\]
where the action of $(\mbC^*)^2$ on $\mbA^6 = \Spec \Cox (P_n)$ is given by the above matrix.
Note that the Weil divisor class group $\Cl (P_n)$ is isomorphic to $\mbZ^2$.
There is a natural morphism $\Pi \colon P \to \mbP^1$ defined as the projection to coordinates $w_0,w_1$, and this realizes $P$ as a weighted projective space bundle over $\mbP^1$ whose fibers are $\mbP (1,1,1,2)$.
For a non-negative integer $n$ and homogeneous polynomials $a \in \mbC [w_0,w_1]$ and $f \in \mbC [x_0,x_1,x_2]$ of degree respectively $2 n$ and $4$, we define
\[
X_n = (a y^2 + f = 0) \subset P_n.
\]
We define $\pi := \Pi|_{X_n} \colon X_n \to \mbP^1$.
Throughout this paper, we assume that $a$ does not have a multiple component.

\begin{Rem} \label{rem:coin}
We explain that $X_n/\mbP^1$ constructed as above coincides with the one given in \cite[Section 3.1]{Ahm1}.
We choose and fix any pair  $b, c \in \mbC [w_0,w_1]$ of homogeneous polynomials of degree $n$ such that $a = b c$ and define
\[
\mcX'_n = (b t^2 + c f = 0) \subset \mbP^1_{w_0,w_1} \times \mbP (1_{x_0},1_{x_1},1_{x_2},2_t).
\]
Let $\pi' \colon \mcX'_n \to \mbP^1$ be the projection to the coordinates $w_0,w_1$.
Then, $(c = t = f = 0) \subset \mcX'_n$ is a disjoint union of $n$-curves $C'_1,\dots,C'_n$ and $\mcX'_n$ has a singularity of type $\mbC \times \frac{1}{2} (1,1)$ along each $C'_i$.
Blowing up $\mcX'_n$ along those curves and then contracting the strict transforms of the $\pi'$-fibers containing $C'_i$, we obtain a birational $\mcX'_n \ratmap \mcX_n$ to the del Pezzo fibration $\mcX_n \to \mbP^1$ constructed in \cite{Ahm1}.

Now we have a birational map $\Psi \colon P_n \ratmap \mbP^1 \times \mbP (1,1,1,2)$ defined by the correspondence $t = c y$.
It is easy to see that $\Psi$ restricts to a birational map $\psi \colon X_n \ratmap \mcX'_n$.
Moreover, it is straightforward to see that $\psi^{-1} \colon \mcX'_n \ratmap X_n$ is obtained by blowing up $\mcX'_n$ along $C'_1,\dots,C'_n$ and then contracting the strict transforms of the fibers containing $C'_i$.
This shows $\mcX_n/\mbP^1 \cong X_n/\mbP^1$.
\end{Rem}

\begin{Rem}
We explain that both $X_0$ and $X_1$ are rational.
If $n = 0$, then we have $X_0 \cong \mbP^1 \times S$, where $S = (y^2 + f = 0) \subset \mbP (1,1,1,2)$ is a (nonsingular) del Pezzo surface of degree $2$, and thus $X_0$ is clearly rational.
Suppose $n = 1$.
Then, as explained in Remark \ref{rem:coin}, $X_1$ is birational to $\mcX'_1 = (b y^2 + c f = 0) \subset \mbP^1 \times \mbP (1,1,1,2)$, where, where $b,c \in \mbC [w_0,w_1]$ are homogeneous polynomials of degree $1$ such that $a = b c$.
It is clear that the projection $\mcX'_1 \ratmap \mbP (1,1,1,2)$ is birational.
Hence $\mcX'_1$ and $X_1$ are rational.
\end{Rem}

In the rest of this section, we show that $\pi \colon X_n \to \mbP^1$ is indeed a del Pezzo fibration for $n \ge 1$.

\begin{Def}
Let $\pi \colon X \to \mbP^1$ be a surjective morphism with connected fibers from a normal projective $3$-fold $X$.
We say that $\pi \colon X \to \mbP^1$ is a {\it del Pezzo fibration} over $\mbP^1$ if the following conditions are satisfied:
\begin{enumerate}
\item $X$ is $\mbQ$-factorial and has only terminal singularities.
\item $-K_X$ is $\pi$-ample.
\item $\rho (X) = 2$.
\end{enumerate}
\end{Def}


\begin{Rem}
We explain the natural affine open subsets of $P_n$ and $X_n$.
Since we will work over an algebraically closed field of characteristic $2$ in the next section, we assume in this remark that the ground field of $P_n$ and $X_n$ is an algebraically closed field $\K$ of arbitrary characteristic.

We denote by $U_{w_i,x_j}$ the open subset $(w_i \ne 0) \cap (x_j \ne 0) \subset P_n$ and by $U_{w_i,y}$ the open subset $(w_i \ne 0) \cap (y \ne 0) \subset P_n$.
We see that $P_n$ is covered by $U_{w_i,x_j}$ and $U_{w_i,y}$ for $i = 0,1$ and $j = 0,1,2$.

We see that $U_{w_0,x_0}$ is an affine $4$-space $\mbA^4$ and the restrictions of $w_1, x_1,x_2,y$ on $U_{w_0,x_0}$ form affine coordinates of $U_{w_0,x_0}$.
Indeed, if we denote by $\tilde{w}_1 = w_1/w_0$, $\tilde{x}_i = x_i/x_0$ for $i = 1,2$ and $\tilde{y} = y w_0^n/x_0^2$, then $U_{w_0,x_0}$ is the affine $4$-space with affine coordinates $\tilde{w}_0,\tilde{x}_1,\tilde{x}_2,\tilde{y}$.
The affine scheme $X_n \cap U_{w_0,x_0}$ is defined by the equation $\tilde{y} a (1,\tilde{w}_1) + f (1,\tilde{x}_1,\tilde{x}_2) = 0$.
The same description applies for the other $U_{w_i,x_j}$.

We see that $U_{w_0,y}$ is the quotient $\mbA^4/\bmu_2$ of $\mbA^4$ by the action of $\bmu_2 = \Spec \K [t]/(t^2)$.
Indeed, if we denote by $\tilde{w}_1 = w_1/w_0$, $\tilde{x}_i = x_i w_0^{n/2}/y^{1/2}$ for $i = 0,1,2$, then $U_{w_0,y}$ is the quotient of $\mbA^4$ with coordinates $\tilde{w}_1, \tilde{x}_0,\tilde{x}_1,\tilde{x}_2$ under the $\bmu_2$-action given by
\[
\tilde{w}_0 \mapsto \tilde{w}_0, \ 
\tilde{x}_i \mapsto \tilde{x}_i \otimes \bar{t},
\]
where $\bar{t} \in \K [t]/(t^2)$.
Here, the above operation defines a ring homomorphism $R \to R \otimes \K [t]/(t^2)$, where $R = \K [\tilde{w}_0,\tilde{x}_0,\tilde{x}_1,\tilde{x}_2]$, and we have $\mbA^4/\bmu_2 = \Spec R^{\bmu_2}$.
When $\K = \mbC$, we can replace $\bmu_2$ with $\mbZ/2 \mbZ$ and the action is given simply by $\tilde{w}_0 \mapsto \tilde{w}_0$ and $\tilde{x}_i \mapsto - \tilde{x}_i$.
The affine scheme $X_n \cap U_{w_0,y}$ is the quotient of the affine scheme defined $a (1,\tilde{w}_1) + f (\tilde{x}_0,\tilde{x}_1,\tilde{x}_2) = 0$ by the $\bmu_2$-action.
The same description applies for $U_{w_1,y}$.

Sometimes we will abuse the notation and say that $U_{w_0,x_0}$ is the affine $4$-space $\mbA^4$ with coordinates $w_1,x_1,x_2,y$ and $X_n \cap U_{w_0,x_0}$ is defined by $y a (1,w_1) + f (1,x_1,x_2) = 0$.
\end{Rem}

\begin{Lem}
$X_n$ is nonsingular outside $(x_0 = x_1 = x_2 = 0) \cap X$ and $X_n$ has a singular point of type $\frac{1}{2} (1,1,1)$ at each point of $(x_0 = x_1 = x_2 = 0) \cap X$.
\end{Lem}

\begin{proof}
Set $U = U_{w_0,x_0}$ which is an affine $4$-space with affine coordinates $w_1,x_1,x_2,y$ and $X_n \cap U$ is defined by $y^2 a_0 +  f_0 = 0$, where $a_0 = a (1,w_1)$ and $f_0 = f (1,x_1,x_2)$.
It is straightforward to see that
\[
\begin{split}
\Sing (X_n \cap U_n) &= \left( y^2 \frac{\prt a_0}{\prt w_1} = \frac{\prt f_0}{\prt x_1} = \frac{\prt f_0}{\prt x_2} = 2 y a_0 = y^2 a_0 + f_0 = 0 \right) \\
&\subset \left( \frac{\prt f_0}{\prt x_1} = \frac{\prt f_0}{\prt x_2} = f_0 = 0 \right) = \emptyset,
\end{split}
\] 
where the last equality holds since $f_0 = f (1,x_1,x_2)$ defines in $\mbA^2$ a nonsingular curve.
By symmetry, we conclude that $X \cap U_{w_i,x_j}$ is nonsingular for $i = 0,1$ and $j = 0,1,2$.
Since the open subsets $U_{w_i,x_j}$ for $i = 0,1$ and $j = 0,1,2$ cover $P_n \setminus (x_0 = x_1 = x_2 = 0)$, we see that $X_n$ is nonsingular outside $(x_0 = x_1 = x_2 = 0) \cap X_n$.

Let $P \in (x_0 = x_1 = x_2 = 0) \cap X_n$.
Then $a (P) = 0$ and we may assume that $w_1$ vanishes at $P$ after replacing $w_0,w_1$.
We work on $U = U_{w_0,y} \cong \mbA^4/\bmu_2$.
We see that $X_n \cap U$ is the quotient of $V := (a (1,w_1) + f = 0) \subset \mbA^4$ by the $\bmu_2$-action and $P$ corresponds to the origin.
Since $a$ vanishes at $P$ and it does not have a multiple component, we have $a (1,w_1) = w_1 + (\text{higher order terms})$, so that $x_0,x_1,x_2$ form local coordinates of $V$ at the origin.
Thus the point $P$ is of type $\frac{1}{2} (1,1,1)$.
\end{proof}

For $n \ge 1$, we construct a birational morphism $\theta \colon X_n \ratmap Z_n$ as follows.
We set $\xi_0 := w_0^n, \xi_1 := w_0^{n-1} w_1, \dots, \xi_n := w_1^n$ and let 
\[
\Theta \colon P_n \to \mbP (1_{x_0},1_{x_1},1_{x_2},2_{y_0},\dots,2_{y_n})
\]
be the toric morphism defined by the correspondence $y_i = y \xi_i$. 
We see that the image of $\Theta$, which we denote by $T_n$, is defined by $h_1 = \cdots = h_N = 0$, where $h_1, \dots, h_N$ are the homogeneous polynomials in $y_0,\dots,y_n$ defining the image of the $n$-ple Veronese embedding $\mbP^1 \inj \mbP^n_{y_0,\dots,y_n}$.
We see that $\Theta \colon P_n \to T_n$ is a birational morphism contracting the divisor $(y = 0) \cong \mbP^1 \times \mbP^2$ to the plane $\Delta := (y_0 = \cdots = y_n = 0) \subset T_n$.
It follows that $T_n$ is a projective simplicial toric variety with Picard number $1$.
The image of $X_n$ under $\Theta$ is a hypersurface $V_n$ in $T_n$ defined by $q + f = 0$, where $q = q (y_0,\dots,y_n)$ is a quadratic polynomial such that $q (y \xi_0,\dots,y \xi_n) = a y^2$.
We see that $\theta := \Theta|_{X_n} \colon X_n \to V_n$ is a birational morphism contracting the divisor $(y = 0) \cap X_n \cong \mbP^1 \times C$ to the curve $\Delta \cap V_n \cong C$, where $C$ is the plane curve defined by $f = 0$.

\begin{Lem} \label{lem:picZ}
If $n \ge 1$, then $V_n$ is a normal projective $\mbQ$-factorial $3$-fold with Picard number $1$.
\end{Lem}

\begin{proof}
Note that $X_n$ is $\mbQ$-factorial since it has only quotient singularities.
It follows that $V_n$ is $\mbQ$-factorial since $\theta$ is an extremal contraction (which is not necessarily $K_{X_n}$-negative).
We see that the singularity of $T_n$ along the plane $\Delta$ is of type $\mbP^2 \times \frac{1}{n} (1,1)$ and $V_n$ intersects $\Delta$ transversally.
Moreover, outside the curve $\Delta \cap V_n$, singular points of $V_n$ are of type $\frac{1}{2} (1,1,1)$.
This implies that $V_n$ is a $V$-submanifold of $T_n$ and thus, by \cite[Proposition 3.5]{BC}, $V_n$ is quasi-smooth in $T_n$.
Here, we refer the reader to \cite[Section 3]{BC} for the definitions of $V$-submanifold and quasi-smoothness.
It then follows from \cite[Proposition 4]{Roan} that $\rho (V_n) = \rho (T_n) = 1$ since $V_n$ is quasi-smooth hypersurface defined by a global section of an ample divisor on $T_n$.
\end{proof}

\begin{Lem}
For $n \ge 1$, the fibration $\pi \colon X_n \to \mbP^1$ is a del Pezzo fibration.
\end{Lem}

\begin{proof}
Assume that $n \ge 1$.
We see that $X$ has only terminal singularities of type $\frac{1}{2} (1,1,1)$ and $\mbQ$-factorial.
By Lemma \ref{lem:picZ}, we have $\rho (X_n) = \rho (V_n) + 1 = 2$ since $\theta \colon X_n \to V_n$ is a birational morphism contracting a prime divisor.
This shows that $\pi$ is an extremal contraction and thus $X_n/\mbP^1$ is indeed a del Pezzo fibration.
%
\end{proof}

\begin{Rem}
The above arguments applies to more general cases without any change.
Let $g \in \mbC [x_0,x_1,x_2]$ be a homogeneous polynomial of degree $4$ such that the plane curve in $\mbP^2$ defined by $g$ is nonsingular.
Then the hypersurface $X_n = (a y^2 + g = 0) \subset P_n$, where $a \in \mbC [w_0,w_1]$ is a homogeneous polynomial of degree $2 n$ which does not have a multiple component, together with the projection $\pi \colon X_n \to \mbP^1$ is a del Pezzo fibration provided that $n \ge 1$.
\end{Rem}

We give a definition of $G$-Mori fiber space.

\begin{Def} \label{def:GMFS}
Let $G$ be a group.
A $G$-{\it Mori fiber space} is a normal projective variety $X$, where $G$ acts faithfully on $X$, together with a $G$-equivariant morphism $\pi \colon X \to S$ onto a normal projective variety $S$ with the following properties.
\begin{enumerate}
\item $X$ is $G \mbQ$-factorial, that is, every $G$-invariant Weil divisor on $X$ is $\mbQ$-Cartier, and $X$ has only terminal singularities.
\item $-K_X$ is $\pi$-ample.
\item $\dim S < \dim X$ and $\pi$ has connected fibers.
\item $\rank \Pic^G (X) - \rank \Pic^G (S) = 1$. 
\end{enumerate}
\end{Def}

We see that the Klein simple group $G := \PSL_2 (\mbF_7)$ acts on $X_n/\mbP^1$ along the fibers, so that $X_n/\mbP^1$ is a $G$-Mori fiber space for $n \ge 1$.
For $n = 0$, $X_0/\mbP^1 \cong S \times \mbP^1/\mbP^1$ is not a del Pezzo fibration.
Nevertheless, we have $\rho^G (X_0) = 1$, so that $X_0/\mbP^1$ is a $G$-Mori fiber space as well. 


\section{Proof of Main Theorem}
\label{sec:proof}

\subsection{Reduction modulo $2$}
\label{sec:redmod2}

In the following, we drop the subscript $n$ and write $P = P_n$, $X = X_n$.
In the previous section, the toric variety $P$ is defined over $\mbC$.
We can define $P$ over an arbitrary field or more generally an arbitrary ring.
For a field or a ring $K$, we denote by $P_K$ the toric variety over $\Spec K$ defined by the same fan as that of $P$.
Then, since $f = x_0^3 x_1 + x_1^3 x_2 + x_2^3 x_0$ is defined over $\mbZ$, we can define subscheme $X_K = (a y^2 + f = 0) \subset P_K$ for a homogeneous polynomial $a \in K [w_0,w_1]$ of degree $2 n$.

We write $a = \alpha_0 w_0^{2 n} + \alpha_1 w_0^{2 n - 1} w_1 + \cdots + \alpha_{2 n} w_1^{2 n}$, where $\alpha_i \in \mbC$.
Assume that $\alpha_0,\dots,\alpha_{2 n}$ are very general so that they are algebraically independent over $\mbZ$.
Then, the ring $\mbZ [\alpha_0,\dots,\alpha_{2 n}]$ is isomorphic to a polynomial ring of $2 n + 1$ variables over $\mbZ$ and the ideal $(2)$ is a prime ideal.
We define
\[
R = \mbZ [\alpha_0,\dots,\alpha_{2 n}]_{(2)}
\]
which is a DVR whose residue field is of characteristic $2$.
We can define $X_R = (a y^2 + f = 0) \subset P_R$, which is a scheme over $\Spec R$ and whose geometric generic fiber is isomorphic to $X_{\mbC}$.

\begin{Lem} \label{lem:redmod2}
Let $\K$ be an algebraically closed field which is uncountable.
If $X_{\K} = (a y^2 + f = 0) \subset P_{\K}$ is not ruled for a very general $a \in \K [w_0,w_1]$, then $X = X_{\mbC} = (a y^2 + f = 0) \subset P_{\mbC}$ is not ruled for a very general $a \in \mbC [w_0,w_1]$. 
\end{Lem}

\begin{proof}
Let $X'$ be the geometric special fiber of $X_R \to \Spec R$ defined over $\K$.
We can write $X' = (a' y^2 + f = 0) \subset P_{\K}$ for some $a' \in \K [w_0,w_1]$ and $a'$ corresponds to a very general element.
By the Matsusaka's theorem \cite[V.1.6 Theorem]{Kol}, if $X'$ is not ruled, then $X$ is not ruled.
This completes the proof.
\end{proof}

\subsection{Koll\'{a}r's technique}
\label{sec:Kol}

In this subsection, we briefly recall Koll\'{a}r's argument of proving non-ruledness of suitable covering spaces in positive characteristic.

We apply the following non-ruledness criterion which is a slight generalization of \cite[V.5.1 Lemma]{Kol}.

\begin{Lem} \label{lem:crinonruled}
Let $Y$ be a smooth proper variety defined over an algebraically closed field and $\mcM$ a big line bundle on $Y$.
If there is an injection $\mcM \inj (\Omega_Y^i)^{\otimes m}$ for some $i > 0$ and $m > 0$, then $Y$ is not separably uniruled.
\end{Lem}

\begin{proof}
Suppose that $Y$ is separably uniruled.
Then, there exists a separable dominant map $\varphi \colon \mbP^1 \times V \ratmap Y$, where $V$ is a normal projective variety.
After shrinking $V$, we may assume that $\varphi$ is a morphism and $Y$ is smooth.
The homomorphism $r^* \Omega_Y^1 \inj \Omega_{V \times \mbP^1}^1$ is an isomorphism on a non-empty open subset since $\varphi$ is separable.
This induces an injection $\varphi^* \mcM^{\otimes k} \inj (\Omega^i_{V \times \mbP^1})^{\otimes m k}$ for any $k \ge 1$.
The invertible sheaf $\mcM$ is big so that the global sections of $\varphi^* \mcM^{\otimes k}$ separates points on a non-empty open subset of $V \times \mbP^1$ for a sufficiently large $k$.
This is a contradiction since the global sections of $(\Omega^i_{V \times \mbP^1})^{\otimes m k}$ do not separate points in a fiber. 
\end{proof}

Our aim is to prove that the variety $X_n$ defined over an algebraically closed field of characteristic $2$ is not ruled.
In view of Lemma \ref{lem:crinonruled}, it is enough to construct a resolution $r \colon Y \to X_n$ and a big line bundle $\mcM$ which is a subsheaf of $(\Omega^i_Y)^{\otimes m}$ for some $m > 0$.
As we will see in Section \ref{sec:biglb}, there is a purely inseparable cover $X_n \to Z$ of degree $2$ for some normal projective variety $Z$.
In the following we explain the Koll\'{a}r's construction of a big line bundle on a nonsingular model of a suitable cyclic covering space in a general setting.
 
Let $Z$ be a variety of dimension $n$ defined over an algebraically closed field $\K$ of characteristic $p > 0$, $\mcL$ a line bundle on $Z$ and $s \in H^0 (Z, \mcL^{\otimes m})$ a global section of $\mcL^m$ for some $m > 0$.
Let $U = \Spec (\oplus_{i \ge 0} \mcL^{- i})$ be the total space of the line bundle $\mcL$ and let $\rho_U \colon U \to Z$ the natural morphism.
We denote by $y \in H^0 (U, \rho_U^* \mcL)$ the zero section and define 
\[
Z [\sqrt[m]{s}] = (y^m - s = 0) \subset U.
\] 
We say that $Z [\sqrt[m]{s}]$ is the {\it cyclic covering of $Z$ obtained by taking $m$-th roots of $s$}.
Set $X = Z [\sqrt[m]{s}]$ and let $\rho = \rho_U|_X \colon X \to Z$ be the cyclic covering.

From now on we assume that $Z$ is nonsingular and $m$ is divisible by $p$.
We have a natural differential $d \colon \mcL^m \to \mcL^m \otimes \Omega_Z^1$ whose construction is given below.
Let $\tau$ be a local generator of $\mcL$ and $t = g \tau^m$ a local section.
Let $x_1,\dots,x_n$ be local coordinates of of $Z$.
Then we define
\[
d (t) := \sum \frac{\prt g}{\prt x_i} \tau^m d x_i.
\]
This is independent of the choices of local coordinates and the local generator $\tau$, and thus defines $d$.
For the section $s \in H^0 (Z, \mcL^m)$, we can view $d (s)$ as a sheaf homomorphism $d (s) \colon \mcO_Z \to \mcL^m \otimes \Omega^1_Z$.
By taking the tensor product with $\mcL^{-m}$, we obtain $d s \colon \mcL^{- m} \to \Omega^1_Z$.

\begin{Def}[{\cite[V.5.8 Definition]{Kol}}] \label{def:Q}
We define
\[
\mcQ (\mcL, s) = \left( \det \mathrm{Coker} (d s) \right)^{\vee \vee}.
\]
\end{Def}

We have $\mcQ (\mcL,s) \cong \mcL^m \otimes \omega_Z$.

\begin{Lem}[{\cite[5.5 Lemma]{Kol}}] \label{lem:injQ}
There is an injection
\[
\rho^* \mcQ (\mcL,s) \inj (\Omega_X^{n-1})^{\vee \vee}.
\]
\end{Lem}

\begin{Rem} \label{rem:eta}
Let $x_1,\dots,x_n$ be local coordinates of $Z$ at a point $P$ and write $s = g \tau^{\otimes m}$ as before.
Then, $\rho^*\mcQ (\mcL,s) \subset (\Omega_X^1)^{\vee \vee}$ is generated by the form
\[
\eta = (\pm) \frac{d x_2 \wedge \cdots \wedge d x_n}{\prt f/\prt x_1}
= (\pm)  \frac{d x_1 \wedge d x_3 \wedge \cdots \wedge d x_n}{\prt f/\prt x_2}
= (\pm) \frac{d x_1 \wedge \cdots \wedge d x_{n-1}}{\prt f/\prt x_n}.
\]
See \cite[V.5.9 Lemma]{Kol} for a detail. 
\end{Rem}

We explain that if the singularity of $X$ is mild, then we can lift $\rho^* \mcQ (\mcL,s)$ to an invertible subsheaf of $\Omega^{n-1}_Y$, where $Y$ is a suitabale nonsingular model of $X$.
For simplicity of description, we assume that $p = 2$ and $n = \dim Z = 3$. 

\begin{Def}[{\cite[V.5.6 Definition]{Kol}, see also \cite[V.5.7 Exercise]{Kol}}]
We say that $s \in H^0 (Z, \mcL^m)$ has a {\it critical point} at $P \in Z$ if $d (s) \in H^0 (Z, \mcL^m \otimes \Omega^1_Z)$ vanishes at $P$.

We say that $s$ has an {\it almost nondegenerate critical point} at $P$ if in suitable choice of local coordinates $x_1,x_2,x_3$, we can write 
\[
g = \alpha x_1^2 + x_2 x_3 + x_1^3 + h,
\]
where $\alpha \in \K$, $s = g \tau^m$ for a local generator $\tau$ of $\mcL$ at $P$, $h = h (x_1,x_2,x_3)$ consists of monomials of degree at least $3$ and it does not involve $x_1^3$.
\end{Def}

\begin{Lem}[{\cite[V.5.10 Proposition]{Kol}}] \label{lem:liftQ}
Suppose that $s$ has only almost nondegenerate critical points.
Then the singularities of $X$ are isolated singularities and they can be resolved by blowing up each singular point of $X$.
Moreover, if we denote by $r \colon Y \to X$ the blowup of each singular points of $X$, then $r^*\rho^*\mcQ (\mcL,s) \inj \Omega^2_Y$.
\end{Lem}

\subsection{Construction of a big line bundle}
\label{sec:biglb}

Throughout this subsection, we work over an algebraically closed field $\K$ of characteristic $2$ which is uncountable.
We write $P = P_{\K}$ and $X = X_{\K}$. 
We do not assume $n \ge 5$ for the moment.
We set $P^{\circ} = P \setminus (x_0 = x_1 = x_2 = 0)$ and $X^{\circ} = X \cap P^{\circ}$.
Note that $P^{\circ}$ is the nonsingular locus of $P$.
We define
\[
Q =
\begin{pmatrix}
w_0 & w_1 & & x_0 & x_1 & x_2 & z \\
1 & 1 & \vdots & 0 & 0 & 0 & - 2 n \\
0 & 0 & \vdots & 1 & 1 & 1 & 4
\end{pmatrix}
\]
and then define $Z$ to be the hypersurface in $Q$ defined by 
\[
z a + f = 0.
\]
Let $\rho \colon X \to Z$ be the morphism which is defined by the correspondence $z = y^2$, which is a purely inseparable finite morphism of degree $2$.

\begin{Lem} \label{lem:cra}
Let $a \in \K [w_0,w_1]$ be a general homogeneous polynomial of degree $2 n$.
Then the set
\[
\Cr (a) := \left( \frac{\prt a}{\prt w_0} = \frac{\prt a}{\prt w_1} = 0 \right) \subset \mbP^1_{w_0,w_1}
\]
consists of finitely many points and $\Cr (a) \cap (a = 0) = \emptyset$.
Moreover, for each $P \in \Cr (a)$, we can choose a local coordinate $w$ of $\mbP^1$ at $P$ such that 
\[
a = \alpha + \beta w^2 + w^3 + (\text{higher order terms})
\]
for some $\alpha,\beta \in \K$ with $\alpha \ne 0$.
\end{Lem}

\begin{proof}
The set $\Cr (a)$ is clearly a finite set of points.
As a generality of $a$, we in particular require that $a$ does not have a multiple component.
It is then clear that $\Cr (a) \cap (a = 0) = \emptyset$.
The last assertion follows by counting dimension.
Let $P \in \mbP^1$ be a point and $w$ a local coordinate of $\mbP^1$ at $P$.
We can write $a = \sum \alpha_i w^i$, $\alpha_i \in \K$. 
We say that $a$ has a bad critical point at $P$ if $\alpha_1 = \alpha_3 = 0$.
Two conditions $\alpha_1 = \alpha_3 = 0$ are imposed for $a$ to have a bad critical point at a given $P \in \mbP^1$.
Since $P$ runs through $\mbP^1$, we see that homogeneous polynomials $a$ which have a bad critical point at some point $P \in \mbP^1$ form at most $2 - 1 = 1$ codimensional subfamily in the space of all $a \in \K [w_0,w_1]$.
Thus, a general $a$ does not have a bad critical point at all and the proof is completed. 
\end{proof}

\begin{Lem} \label{lem:crf}
The set
\[
\Cr (f) := \left( \frac{\prt f}{\prt x_0} = \frac{\prt f}{\prt x_1} = \frac{\prt f}{\prt x_2} = 0 \right) \subset \mbP^2_{x_0,x_1,x_2}
\]
consist of finitely many closed points and $\Cr (f) \cap (f = 0) = \emptyset$.
Moreover, for each $P \in \Cr (f)$, we can choose local coordinates $t_1, t_2$ of $\mbP^2$ at $P$ such that
\[
f = \gamma + t_1 t_2 + (\text{higher order terms})
\]
for some $\gamma \ne 0$.
\end{Lem}

\begin{proof}
We have
\[
\frac{\prt f}{\prt x_0} = x_0^2 x_1 + x_2^3, \ 
\frac{\prt f}{\prt x_1} = x_0^3 + x_1^2 x_2, \ 
\frac{\prt f}{\prt x_2} = x_1^3 + x_2^2 x_0.
\]
By a straightforward computation, we have
\[
\Cr (f) = \{ (1\!:\!\zeta^{3 i}\!:\!\zeta^i) \mid 0 \le i \le 6 \},
\]
where $\zeta \in \K$ is a primitive $7$th root of unity.
It is also straightforward to see $f (P) \ne 0$ for $P \in \Cr (f)$.
For the last assertion, we work on the affine open subset $U = (x_0 \ne 0) \subset \mbP^2$.
Note that $\Cr (f) \subset U$.
By setting $x_0 = 1$, we think of $x_1,x_2$ as affine coordinates of $U \cong \mbA^2$.
We have $f = x_1 + x_1^3 x_2 + x_2^3$ on $U$.
For the verification of the last assertion, it is enough to show that the Hessian of $f$ at $P = (1\!:\!\zeta^{3 i}\!:\!\zeta^i) \in \Cr (f)$ is non-zero.
We have $\prt^2 f/\prt x_1^2 = 1$, $\prt^2 f/\prt x_1 \prt x_2 = x_1^2$ and $\prt^2 /\prt x_2^2 = 0$, so that we can compute the Hessian as
\[
\left| \begin{array}{cc}
1 & x_1^2 \\
x_1^2 & 0
\end{array} \right| (P)
= \zeta^{12 i} \ne 0.
\] 
Therefore, the last assertion is proved.
\end{proof}

We set $Q^{\circ} = Q \setminus (x_0 = x_1 = x_2 = 0)$ and $Z^{\circ} = Z \cap Q^{\circ}$.

\begin{Lem} \label{lem:nonsingZcirc}
The quasi projective variety $Z^{\circ}$ is nonsingular.
\end{Lem}

\begin{proof}
We work on the open subset $U = U_{w_0,x_0} \subset Q$ which is an affine $4$-space with coordinates $w_1,x_1,x_2,z$.
We see that $Z \cap U$ is defined by $a_0 z + f_0 = 0$, where $a_0 = a (1,w_1)$ and $f_0 = f (1,x_1,x_2)$.
We have
\[
\Sing (Z \cap U) = \left( z \frac{a_0}{\prt w_1} = \frac{\prt f_0}{\prt x_1} = \frac{\prt f_0}{\prt x_1} = a_0 = f_0 = 0 \right) = \emptyset,
\]
where the last equality follows since the curve $f_0 = 0$ in $\mbA^2$ is nonsingular.
By symmetry $Z \cap U_{w_i,x_j}$ is nonsingular for $i = 0,1$ and $j = 0,1,2$.
Since $Z^{\circ}$ is covered by $U_{w_i,x_j}$ for $i = 0,1$ and $j = 0,1,2$, the proof is completed. 
\end{proof}

Let $H_Q$ and $F_Q$ be the divisor class on $Q$ which corresponds to the weight ${}^t (0 \ 1)$ and ${}^t (1 \ 0)$, respectively, that is, $F_Q$ is the fiber class of the projection $Q \to \mbP^1$ and $H_Q|_{F_Q} \in |\mcO_{\mbP (1,1,1,4)} (1)|$.
We set $H_Z = H_Q|_Z$ and $F_Z = F_Q|_Z$.
We define $\mcL$ to be the sheaf $\mcO_Z (2 H_Z - n F_Z)$ whose restriction on $Z^{\circ}$ is an invertible sheaf.
Note that we have $z \in H^0 (Z,\mcL^2)$.
It is clear that $X \cong Z [\sqrt{z}]$.
In the following we choose and fix a general $a \in \K [w_0,w_1]$ so that the assertions of Lemma \ref{lem:cra} hold.

\begin{Lem} \label{lem:crz}
The section $z \in H^0 (Z^{\circ}, \mcL^2)$ has only almost nondegenerate critical points on $Z^{\circ}$.
\end{Lem}

\begin{proof}
Let $\Cr (z) \subset Z^{\circ}$ be the set of critical points of $z$.
Since
\[
\frac{\prt (a z + f)}{\prt z} = a,
\]
$z$ can be chosen as a part of local coordinates at every point $P \in Z^{\circ}$ such that $a (P) = 0$.
It follows that $z$ does not have a critical point at any point $P \in X \cap (a = 0)$.
We work on an open set $U \subset Z^{\circ}$ on which $a \ne 0$ and prove that $z|_U$ has only almost nondegenerate critical points on $U$.
Since $z = - f/a$ on $U$ and $a$ is a unit on $U$, it is enough to show that $-a^2 z = a f$ has only almost nondegenerate critical points on $U$.
Let $P \in U$ be a critical point of $z$.
We have
\[
\frac{\prt (a f)}{\prt w_i} = \frac{\prt a}{\prt w_i} f, \ 
\frac{\prt (a f)}{\prt x_j} = a \frac{\prt f}{\prt x_j},
\]
for $i = 0,1$ and $j = 0,1,2$. 
Since $a (P) \ne 0$, we have $(\prt f/\prt x_j) (P) = 0$ for $j = 0,1,2$.
By Lemma \ref{lem:crf}, we have $f (P) \ne 0$, which implies $(\prt a/\prt w_i) (P) = 0$ for $i = 0,1$.
By Lemmas \ref{lem:cra} and \ref{lem:crf}, we can choose local coordinates $w, t_1, t_2$ of $Z$ at $P$ such that
\[
a f = (\alpha + \beta w^2 + w^3 + \cdots)(\gamma + t_1 t_2 + \cdots)
= \alpha \gamma + \beta \gamma w^2 + \alpha t_1 t_2 + \gamma w_1^3 + h,
\] 
where $\alpha, \beta, \gamma \in \K$ with $\alpha,\gamma \ne 0$, $h = h (w,t_1,t_2)$ consists of monomials of degree at least $3$ and it does not involve $w^3$.
This shows that $z$ has only almost nondegenerate critical points on $Z^{\circ}$.
\end{proof}

We define $\mcQ^{\circ} = \mcQ (\mcL, z)|_{Z^{\circ}}$ which is an invertible sheaf on $Z^{\circ}$.
By Lemma \ref{lem:injQ}, we have $\rho^*\mcQ^{\circ} \inj (\Omega_{X^{\circ}}^2)^{\vee \vee}$, where $\rho \colon X^{\circ} = Z^{\circ} [\sqrt{z}] \to Z^{\circ}$.
By adjunction, we have $\omega_Z \cong \mcO_Z (- 3 H_Z + (2 n-2) F_Z)$, hence $\mcQ^{\circ} \cong \mcO_{Z^{\circ}} (H_Z - 2 F_Z)$.
Let $H_P$ and $F_P$ be the divisors on $P$ which correspond to ${}^t (0 \ 1)$ and $(1 \ 0)$, respectively, so that $F_P$ is the fiber class of $\Pi \colon P \to \mbP^1$ and $H_P|_{F_P} \in |\mcO_{\mbP (1,1,1,2)} (1)|$.
We set $H = H_P|_X$ and $F = F_P|_X$.
We have $H = \rho^*H_Z$ and $F = \rho^*F_Z$, hence $\rho^*\mcQ^{\circ} \cong \mcO_{X^{\circ}} (H - 2 F)$.
Let $\iota \colon X^{\circ} \inj X$ be the open immersion.
The sheaf $\iota_* \rho ^* \mcQ^{\circ} \cong \mcO_X (H - 2 F)$ is a reflexive sheaf of rank $1$ but is not invertible at each singular point of type $\frac{1}{2} (1,1,1)$.
We define $\mcM := \iota_* \rho^* {\mcQ^{\circ}}^2 \cong \mcO_X (2 H - 4 F)$ which is an invertible sheaf on $X$ and we have an injection $\mcM \inj ((\Omega_X^2)^{\otimes 2})^{\vee \vee}$.

Note that $X$ has two kinds of singularities both of which are isolated: one of them are the singular points on $X^{\circ}$ corresponding to the critical points of $z$ and the other ones are singular points of type $\frac{1}{2} (1,1,1)$. 
Let $r \colon Y \to X$ be the blowup of $X$ at each singular points.
By Lemmas \ref{lem:liftQ} and \ref{lem:liftQ}, $Y$ is nonsingular and we have an injection $r^*\mcM |_{Y^{\circ}} \inj (\Omega_{Y^{\circ}}^2)^{\otimes 2}$ on the open subset $Y^{\circ} = r^{-1} (X^{\circ})$.
We will show that there is an injection $r^*\mcM \inj (\Omega^2_Y)^{\otimes 2}$.

\begin{Lem} \label{lem:inj}
There is an injection $r^*\mcM \inj (\Omega_Y^2)^{\otimes 2}$.
\end{Lem}

\begin{proof}
Let $P$ be a singular point of type $\frac{1}{2} (1,1,1)$.
Since we know that $r^*\mcM \inj (\Omega^2_Y)^{\otimes 2}$ on the open subset $Y^{\circ} = r^{-1} (X^{\circ})$, it is enough to show that $r^*\mcM \inj (\Omega_Y^2)^{\otimes 2}$ locally around the exceptional divisor of $r \colon Y \to X$ over $P$.
We can write $\mcM_P = \mcO_{X,P} \! \cdot \! \eta$ for some local section $\eta$ of $((\Omega_{X}^2)^{\otimes 2})^{\vee \vee}$ since $\mcM \subset ((\Omega_{X}^2)^{\otimes 2})^{\vee \vee}$ is an invertible sheaf.
We will show that
\[
\eta = g \left( \frac{d h_1 \wedge d h_2}{h_1 h_2} \right)^{\otimes 2},
\]
for some $g, h_1,h_2 \in \mcO_{X,P}$, and then we will show that $r^*\eta$ does not have a pole along the exceptional divisor over $P$.

After replacing $w_0,w_1$, we assume that $w_1$ vanishes at $P$ (so that $w_0$ does not vanish at $P$).
We work on an open subset $U$ of $U_{w_0,x_0} \subset P$.
Shrinking $U$, we assume $a \ne 0$ on $U$.
Then we have $z = - f/a \in \mcO_U$.
Let $\tilde{w}_1 = w_1/w_0, \tilde{x}_1 = x_1/x_0, \tilde{x}_2 = x_2/x_0$ be the restrictions of $w_1,x_1,x_2$ to $U_{w_0,x_0}$.  
Then, in view of Remark \ref{rem:eta}, after further shrinking $U$, we see that $\mcM |_U$ is generated by
\[
\left( \frac{d \tilde{x}_1 \wedge d \tilde{x}_2}{\prt (-f/a) / \prt x_2} \right)^{\otimes 2}.
\]
In particular, we have
\[
\mcM \otimes K (X) = K (X) \! \cdot \! (d \tilde{x}_1 \wedge d \tilde{x}_2)^{\otimes 2} \subset \left( \Omega^2_{K(X)} \right)^{\otimes 2},
\]
where $K(X)$ is the function field of $X$.

We set $\xi_i = x_i/y^{1/2}$ for $i = 0,1,2$.
Then $\xi_0,\xi_1,\xi_2$ can be chosen as local coordinates of the orbifold chart of $(X,P)$.
Now we have $\tilde{x}_i = \xi_1/\xi_0$ for $i = 1,2$, hence
\[
d \tilde{x}_1 \wedge d \tilde{x}_2 
= d (\xi_1/\xi_0) \wedge d (\xi_2/\xi_0)
= \frac{d (\xi_1 \xi_0^3) \wedge d (\xi_2 \xi_0^3)}{\xi_0^6}.
\]
Here, since the ground field is of characteristic $2$ and $\xi_0^2 \in \mcO_{X,P}$, we have the equality 
\[
d (\xi_i \xi_0^3) = d \left( (\xi_0^2)^2  \frac{\xi_i}{\xi_0} \right) = \xi_0^4 d \left(\frac{\xi_i}{\xi_0}\right)
\]
for $i = 1,2$.
Thus $\mcM_P \otimes K (X) = K (X) \! \cdot \! (d h_1 \wedge d h_2)^{\otimes 2}$, where $h_i = \xi_i \xi_0^3$.
It follows that 
\[
\eta = g \left( \frac{d h_1 \wedge d h_2}{h_1 h_2} \right)^{\otimes 2}
\]
for some rational function $g$.
By \cite{Okada}, we see that $g = \xi_0^2 \xi_1^2 \xi_2^2 h$ for some $h \in \mcO_{X,P}$.

Now, by shrinking $X$, we assume that $r \colon Y \to X$ is the blowup  (more precisely, the weighted blowup with weight $\frac{1}{2} (1,1,1)$) at $P$.
Then the order of the pole of the rational $2$-form
\[
r^* \left(\frac{d h_1 \wedge d h_2}{h_1 h_2}\right)^{\otimes 2}
\]
along the exceptional divisor $E$ is at most $2$ (in fact, an explicit computation shows that the above form does not have a pole along $E$ but we do not need this strong estimate).
It is clear that $r^*\xi_i^2$ vanishes along $E$ to order $1$ so that $r^* g$ vanishes along $E$ to order at least $3$.
Therefore, $r^*\eta$ does not have a pole along $E$ and we have an injection $r^*\mcM \inj (\Omega_Y^2)^{\otimes 2}$.
\end{proof}

\begin{Lem} \label{lem:big}
If $n \ge 5$, then the invertible sheaf $\mcM$ is big.
\end{Lem}

\begin{proof}
Let $m$ be a positive integer such that $m > n/(n-4)$.
We show that the complete linear system of $\mcM^m \cong \mcO_X (2 m H - 4 m F)$ defines a birational map.
We set $k = (n-4) m$ and $l = (n-4) m - n$ which are positive integers.
We see that
\[
\begin{split}
& \{ y^m w_0^k, y^m w_0^{k-1} w_1, \dots, y^m w_1^k\} \cup \{ y^{m-1} w_i^l x_{j_1} x_{j_2} \mid 0 \le i \le 1, 0 \le j_1, j_2 \le 2\}
\end{split}
\]
is a set of sections of $\mcM^m$ and they define a generically finite map.
Indeed, the restriction of sections $y^m w_0^k, y^m w_0^{k-1} w_1$ and $y^{m-1}w_0^l x_{j_1} x_{j_2}$ for $0 \le j_1, j_2 \le 2$ on $X \cap U_{w_0,y}$ are $1, w_1$ and $x_i^2$ for $0 \le j_1,j_2 \le 2$ and they clearly define a generically finite map (in fact an isomorphism).
It follows that the complete linear system of $\mcM^m$ defines a generically finite map and thus $\mcM$ is big.
\end{proof}

\begin{proof}[Proof of Theorem \ref{thm:main}]
Assume that $n \ge 5$.
By Lemmas \ref{lem:inj}, \ref{lem:big} and \ref{lem:crinonruled}, a very general $X_n$ defined over $\K$ is not separably uniruled.
In particular it is not ruled.
Then a very general $X_n$ defined over $\mbC$ is not ruled by Lemma \ref{lem:redmod2} and the proof is completed.
\end{proof}


\end{document}